\documentclass[final,leqno]{siamltex704}
\usepackage{amsmath}
\usepackage{amssymb}
\usepackage{graphicx}
\usepackage[notcite,notref]{showkeys}
\usepackage{tikz}
 \numberwithin{dummy}{section}

\newtheorem{algorithm}{Weak Galerkin Algorithm}
\newtheorem{remark}{Remark}

\newcommand{\bQ}{{\bf Q}}
\newcommand{\bu}{{\bf u}}
\newcommand{\bq}{{\bf q}}
\newcommand{\bw}{{\bf w}}

\newcommand{\be}{{\bf e}}
\newcommand{\bv}{{\bf v}}
\newcommand{\bH}{{\textbf{\textit{H}}}}
\newcommand{\bL}{{\textbf{\textit{L}}}}
\newcommand{\bpsi}{{\boldsymbol\psi}}
\newcommand{\bepsilon}{{\boldsymbol\epsilon}}

\def\Q{{\mathbb Q}}
\def\T{{\mathcal T}}
\def\E{{\mathcal E}}

\def\pT{{\partial T}}
\def\l{{\langle}}
\def\r{{\rangle}}

\def\T{{\mathcal T}}
\def\E{{\mathcal E}}

\def\bbf{{\bf f}}
\def\bg{{\bf g}}
\def\bn{{\bf n}}
\def\bq{{\bf q}}

\def\3bar{{|\hspace{-.02in}|\hspace{-.02in}|}}

  \def\b#1{\mathbf{#1}} 
\def\a#1{\begin{align*}#1\end{align*}}  
 
\def\p#1{\begin{pmatrix}#1\end{pmatrix}}

\title{A conforming discontinuous Galerkin finite element method 
  for the Stokes problem on polytopal meshes}
\author{Xiu Ye\thanks{Department of
Mathematics, University of Arkansas at Little Rock, Little Rock, AR
72204 (xxye@ualr.edu). This research was supported in part by
National Science Foundation Grant DMS-1620016.}
\and
Shangyou Zhang\thanks{Department of
Mathematical Sciences, University of Delaware, Newark, DE 19716 (szhang@udel.edu).}
}
\begin{document}
\maketitle

\begin{abstract}
A new discontinuous Galerkin finite element method for the Stokes equations is developed in the primary velocity-pressure formulation.
This method employs discontinuous polynomials for both velocity and pressure on general polygonal/polyhedral meshes. Most finite element
methods with discontinuous approximation have one or more stabilizing terms for velocity and for pressure to guarantee stability and convergence. This new finite element method has the standard conforming finite element formulation, without any velocity or pressure stabilizers. Optimal-order error estimates are
established for the corresponding numerical approximation in various
norms. The numerical examples are tested for low and  high order elements up to the degree four in 2D and 3D spaces.
\end{abstract}


\pagestyle{myheadings}

\section{Introduction}
Consider the Stokes problem: find the velocity $\bu$ and the
pressure $p$ such that
\begin{eqnarray}
-\Delta\bu+\nabla p&=& \bbf\quad
\mbox{in}\;\Omega,\label{moment}\\
\nabla\cdot\bu &=&0\quad \mbox{in}\;\Omega,\label{cont}\\
\bu&=&0\quad \mbox{on}\; \partial\Omega,\label{bc}
\end{eqnarray}
where $\Omega$ is a polygonal or polyhedral domain in
$\mathbb{R}^d\; (d=2,3)$.

In a conforming finite element method for solving above Stokes equations in primary variables
  \cite{gr,Scott-V,Zhang-HCT,Zhang-PS2,Zhang-Cube,Zhang-PS3}, such as the Taylor-Hood element,
  the velocity $\bu$ is approximated by continuous piecewise polynomials of degree $k$ and 
  the pressure $p$ by continuous/discontinuous piecewise polynomials of degree $k-1$,
  with the following formulation:  Find $\bu_h \in V_h\subset \bH^1_0(\Omega)$ and 
  $p_h \in W_h\subset L^2_0(\Omega)$ such that
\begin{align} (\nabla \bu_h, \nabla \bv_h) - (\nabla\cdot \b v_h, p_h) & = (\bbf, \bv_h)
          \quad \forall \bv_h \in V_h, \label{c-1} \\
             (\nabla\cdot \b u_h, q_h) & = 0
          \qquad \forall q_h \in W_h.\label{c-2} 
\end{align}

In a discontinuous Galerkin finite element method \cite{Zhang},
   the velocity $\bu$ is also approximated by piecewise polynomials of degree $k$ but discontinuous, and 
  the pressure $p$ by discontinuous piecewise polynomials of degree $k-1$ (or $k$)
  on polygonal/polyhedral meshes,
  with the following formulation: 
   Find $\bu_h \in V_h\subset \bL^2(\Omega)$ and 
  $p_h \in W_h\subset L^2_0(\Omega)$ such that
\begin{align} \sum_{T\in\T_h} \big[(\nabla \bu_h, \nabla \bv_h)_T -(p_h, \nabla \cdot
    \bv_h)_T\big]
     \qquad \qquad\qquad & \nonumber\\
   + \sum_{e\in\E_h} \Big( \int_e \{\nabla \bu_h\}\bn \cdot [\bv_h]
   + \epsilon^* \int_e \{\nabla \bv_h\}\bn \cdot [\bu_h]\qquad\qquad\qquad& \nonumber \\
    + \int_e \{p_h\}\bn \cdot [\bv_h] + \frac{\sigma_e}{h}
      \int_e [\bu_h] \cdot [\bv_h] \Big)
        = (\bbf, \bv_h)
          \quad &\forall \bv_h \in V_h, \label{d-1} \\
          \sum_{T\in\T_h}  (q_h, \nabla \cdot
    \bu_h)_T 
   + \sum_{e\in\E_h} \Big( \int_e \{q_h\}\bn \cdot [\bu_h]
  + h \int_e[p_h][q_h]\Big)  = 0
          \quad &\forall q_h \in W_h.\label{d-2} 
\end{align}
It is proved that the pressure stabilizer ($h \int_e[p_h][q_h]$) can be omitted in \eqref{d-2}
   on triangular/tetrahedral meshes.
This simplifies the discontinuous Galerkin finite element formulation.
We would simplify further the formulation \eqref{d-1}-\eqref{d-2} by dropping both 
  stabilizers and all boundary integral terms,  on general polygonal/polyhedral meshes.

In a conforming discontinuous Galerkin finite element method, the original weak
  formulation \eqref{c-1}-\eqref{c-2} of the continuous Galerkin finite element is kept.
But the gradient $\nabla u_h$ of a discontinuous, piecewise polynomial $u_h$ is no long 
a Lebesgue measurable function.
It can be represented as a function in a dual space, $\big(\prod_{T\in \T_h} \bH^1(T) \big)'$.
We define the $L^2$ projection of this function in a piecewise polynomial subspace as 
   a weak gradient, i.e., $\nabla_w u_h\in \prod_{T\in \T_h}  P_k(T) ^d $ such that
\begin{align*}  (\nabla_w u_h, \bq) 
       & =  (\nabla u_h, \bq) \quad \forall \bq \in \prod_{T\in \T_h}  P_j(T) ^d,
\end{align*} where $j\le k+n+d-3$ (for $n$-faced polygons/polyhedrons) and 
\begin{align*}  (\nabla u_h, \bq) 
       & = \sum_{T\in \T_h} (u_h, -\nabla\cdot \bq)_T 
                   +\l \{u_h\}, \bq\cdot \bn\r_\pT \quad \forall
             \bq \in \prod_{T\in \T_h}  \bH^1(T).
\end{align*} Such a method has been developed for Poisson equations \cite{Feng,yz1,yz},
 and for biharmonic equations \cite{Cui,yz-b}.
The definition of weak gradient comes from the weak Galerkin finite element method 
 \cite{cww,cfx,lllc,wy-stokes,wy,wymix,wwzz,zzm} and the modified weak Galerkin method
  \cite{mwy,tzz,wmgm}. A disadvantage of the conforming discontinuous Galerkin 
  finite element method is its computation of the gradient by higher order polynomials.
But this computation is done locally for basis functions only, in advance, 
    before generating and 
   solving the resulting linear systems of equations.
It is equivalent to using high-order quadrature formula in the continuous finite element.

In this paper, we propose a new finite element method for the Stokes equations with discontinuous approximations on general polytopal meshes. Our new finite element method uses totally discontinuous $kth$ degree polynomial for velocity and $(k-1)th$ degree polynomial for pressure. When discontinuous polynomials are employed for both velocity and pressure, stabilizers for velocity or pressure are normally required for the stability of the corresponding finite element formulations, such as \eqref{d-1}-\eqref{d-2}. 
But in this new conforming discontinuous Galerkin finite element method,  we do not have
 any boundary stabilizer term, neither any other boundary integral term. 
That is, we find $\bu_h \in V_h\subset \bL^2(\Omega)$ and 
  $p_h \in W_h\subset L^2_0(\Omega)$ such that
\begin{align}  (\nabla_w \bu_h, \nabla_w \bv_h)_T -(p_h, \nabla_w \cdot \bv_h)
       & = (\bbf, \bv_h)
          \quad \forall \bv_h \in V_h, \label{new-1} \\
          (q_h, \nabla_w \cdot \bu_h)_T    & = 0
          \qquad \forall q_h \in W_h.  \label{new-2}
\end{align}
To the best of our knowledge, our new method is the only finite element formulation without any velocity or pressure stabilizers among all the methods for the Stokes problem
in primary velocity-pressure form with discontinuous approximations on polytopal meshes.

Optimal order error estimates for the finite element approximations are derived in  energy norm for the velocity, and $L^2$ norm for both the velocity and the pressure.  Numerical examples are tested for the finite elements with different degrees up to $P_4$ polynomials and for different dimensions, 2D and 3D.

\section{ Finite Element Method}
We use standard definitions for the Sobolev spaces $H^s(D)$ and their associated
inner products $(\cdot, \cdot)_{s,D}$, norms $\|\cdot\|_{s,D}$, and seminorms $|\cdot|_{s,D}$ for $s\ge 0$.
When $D = \Omega$, we drop the subscript $D$ in the norm and inner product notation.

Let ${\cal T}_h$ be a partition of the domain $\Omega$ consisting of
polygons in two dimensional space or polyhedra in three dimensional space satisfying
a set of conditions specified in \cite{wymix}. Denote by ${\cal E}_h$ the set
of all  flat faces in ${\cal T}_h$, and let ${\cal
E}_h^0={\cal E}_h\backslash\partial\Omega$ be the set of all
interior faces.

For $k\ge 1$ and given $\T_h$, define two finite element spaces,  for approximating velocity
\begin{eqnarray}
V_h &=&\left\{ \bv\in \bL^2(\Omega):\ \bv|_{T}\in [P_{k}(T)]^d, \;\;\forall T\in\T_h \right\}\label{vh}
\end{eqnarray}
and for approximating pressure
\begin{equation}
W_h =\left\{q\in L_0^2(\Omega): \ q|_T\in P_{k-1}(T),\;\;\forall T\in\T_h\right\}.\label{wh}
\end{equation}

Let $T_1$ and $T_2$ be two elements in $\T_h$
sharing $e\in\E_h$.  For $e\in\E_h$ and $\bv\in V_h+\bH_0^1(\Omega) $, the jump $[\bv]$ is defined as
\begin{equation}\label{jump}
[\bv]=\bv\quad {\rm if} \;e\subset \partial\Omega,\quad [\bv]=\bv|_{T_1}-\bv|_{T_2}\;\; {\rm if} \;e\in\E_h^0.
\end{equation}
The order of $T_1$ and $T_2$ is not essential.

For $e\in\E_h$ and $\bv\in V_h+\bH_0^1(\Omega)$, the average $\{v\}$ is defined  as
\begin{equation}\label{avg}
\{\bv\}=0\quad {\rm if} \;e\subset \partial\Omega,\quad \{\bv\}=\frac12(\bv|_{T_1}+\bv|_{T_2})\;\; {\rm if} \;e\in\E_h^0,
\end{equation}

For a function $\bv\in V_h+\bH_0^1(\Omega)$, its weak gradient $\nabla_w\bv$ is a piecewise polynomial tensor such that $\nabla_w\bv \in \prod_{T\in\T_h} [P_{j}(T)]^{d\times d}$  and  satisfies the following equation,
\begin{equation}\label{wg}
(\nabla_w\bv,\  \tau)_T =  (\bv,\  -\nabla \cdot \tau)_T+
\l\{\bv\}, \ \tau\cdot\bn \r_\pT\quad\forall\tau\in [P_{j}(T)]^{d\times d} 
\end{equation}  on each $T\in\T_h$, 
and  its weak divergence $\nabla_w\cdot\bv$ is a piecewise polynomial such that $\nabla_w\cdot\bv \in  \prod_{T\in\T_h}  P_{k-1}(T)$  and  satisfies the following equation,
\begin{equation}\label{wd}
( \nabla_{w}\cdot\bv,\  q)_T =  ( \bv,\ -\nabla q)_T+\l
  \{\bv\}\cdot\bn,\ q\r_\pT\quad\forall q\in P_{k-1}(T) 
\end{equation} on each $T\in\T_h$.

\begin{remark}
The choice of $j$ in (\ref{wg}) depends on the number of sides/faces of polygon/polyhedron. For triangular mesh, we can choose  $j=k+1$ \cite{aw}. In general, $j=n+k-1$, where $n$ is the number of edges of polygon \cite{yz}.
\end{remark}

\medskip

Then we have the following simple penalty free finite element scheme.

\medskip

\begin{algorithm}
A numerical approximation for (\ref{moment})-(\ref{bc}) can be
obtained by seeking $\bu_h\in V_h$ and $p_h\in W_h$ such that for all $\bv\in V_h$ and $q\in W_h$,
\begin{eqnarray}
(\nabla_w\bu_h,\ \nabla_w\bv)-(\nabla_w\cdot\bv,\;p_h)&=&(f,\;\bv),\label{wg1}\\
(\nabla_w\cdot\bu_h,\;q)&=&0.\label{wg2}
\end{eqnarray}
\end{algorithm}

Let $\Q_h$, $\bQ_h$  and $Q_h$ be the element-wise defined $L^2$ projections onto the local spaces $[P_{j}(T)]^{d\times d}$, $[P_{k}(T)]^{d}$  and $P_{k-1}(T)$ for $T\in\T_h$, respectively.

\begin{lemma}
Let $\boldsymbol\phi\in \bH_0^1(\Omega)$, then on  $T\in\T_h$
\begin{eqnarray}
\nabla_w \boldsymbol\phi &=&\Q_h\nabla\boldsymbol\phi,\label{key}\\
\nabla_w\cdot \boldsymbol\phi &=&Q_h\nabla\cdot\boldsymbol\phi.\label{key1}
\end{eqnarray}
\end{lemma}
\begin{proof}
Using (\ref{wg}) and  integration by parts, we have that for
any $\tau\in [P_{j}(T)]^{d\times d}$
\begin{eqnarray*}
(\nabla_w \boldsymbol\phi,\tau)_T &=& -(\boldsymbol\phi,\nabla\cdot\tau)_T
+\langle \{\boldsymbol\phi\},\tau\cdot\bn\rangle_{\pT}\\
&=& -(\boldsymbol\phi,\nabla\cdot\tau)_T
+\langle \boldsymbol\phi,\tau\cdot\bn\rangle_{\pT}\\
&=&(\nabla \boldsymbol\phi,\tau)_T=(\Q_h\nabla\phi,\tau)_T,
\end{eqnarray*}
which implies the desired identity (\ref{key}). Similarly, we can prove (\ref{key1}).
\end{proof}

For any function $\varphi\in H^1(T)$, the following trace
inequality holds true (see \cite{wymix} for details):
\begin{equation}\label{trace}
\|\varphi\|_{e}^2 \leq C \left( h_T^{-1} \|\varphi\|_T^2 + h_T
\|\nabla \varphi\|_{T}^2\right).
\end{equation}

\section{Well Posedness}
We start this section by introducing two semi-norms $\3bar \bv\3bar$ and  $\|\bv\|_{1,h}$
   for any $\bv\in V_h$ as follows:
\begin{eqnarray}
\3bar \bv\3bar^2 &=& \sum_{T\in\T_h}(\nabla_w\bv,\nabla_w\bv)_T, \label{norm1}\\
\|\bv\|_{1,h}^2&=&\sum_{T\in \T_h}\|\nabla \bv\|_T^2+\sum_{e\in\E_h}h_e^{-1}\|[\bv]\|_{e}^2.\label{norm2}
\end{eqnarray}

It is easy to see that $\|\bv\|_{1,h}$ defines a norm in $V_h$.
But $\3bar  \cdot \3bar$ also defines a norm in $V_h$, by  
the following norm equivalence which has been proved in  \cite{aw,yz},
   to each component of $\bv$.
\begin{equation}\label{happy}
C_1\|\bv\|_{1,h}\le \3bar \bv\3bar\le C_2 \|\bv\|_{1,h} \quad\forall \bv\in V_h.
\end{equation}

The inf-sup condition for the finite element  formulation (\ref{wg1})-(\ref{wg2}) will be derived in the following lemma.

\smallskip
\begin{lemma}\label{Lemma:inf-sup}
There exists a positive constant $\beta$ independent of $h$ such that for all $\rho\in W_h$ and $h$ small enough,
\begin{equation}\label{inf-sup}
\sup_{\bv\in V_h}\frac{(\nabla_w\cdot\bv,\rho)}{\3bar\bv\3bar}\ge \beta
\|\rho\|.
\end{equation}
\end{lemma}

\begin{proof}
For any given $\rho\in W_h\subset L_0^2(\Omega)$, it is known
\cite{gr} that there exists a function
$\tilde\bv\in \bH_0^1(\Omega)$ such that
\begin{equation}\label{c-inf-sup}
\frac{(\nabla\cdot\tilde\bv,\rho)}{\|\tilde\bv\|_1}\ge C_0\|\rho\|,
\end{equation}
where $C_0>0$ is a constant independent of $h$. By
setting $\bv=\bQ_h\tilde{\bv}\in V_h$, we claim that the following
holds true
\begin{equation}\label{m9}
\3bar\bv\3bar\le C\|\tilde{\bv}\|_1.
\end{equation}
It follows from (\ref{happy}) and $\tilde{\bv}\in \bH_0^1(\Omega)$,
\begin{eqnarray*}
\3bar \bv\3bar^2&\le& C\|\bv\|_{1,h}^2=C(\sum_{T\in \T_h}\|\nabla \bv\|_T^2+\sum_{e\in\E_h}h_e^{-1}\|[\bv]\|_{e}^2)\\
&\le&C \sum_{T\in \T_h}\|\nabla \bQ_h\tilde{\bv}\|_T^2+\sum_{e\in\E_h}h_e^{-1}\|[\bQ_h\tilde{\bv}-\tilde{\bv}]\|_{e}^2\\
&\le& C\|\tilde{\bv}\|_1^2,
\end{eqnarray*}
which implies the  inequality (\ref{m9}).
It follows from (\ref{wd}) and (\ref{c-inf-sup}) that
\begin{eqnarray*}
|(\nabla_w\cdot\bv,\;\rho)_{\T_h}|&=&|-(\bv,\;\nabla\rho)_{\T_h}+\l\{\bv\},\rho\bn\r_{\partial\T_h}|\\
&=&|-(\tilde\bv,\;\nabla\rho)_{\T_h}+\l\{\bQ_h\tilde\bv\},\rho\bn\r_{\partial\T_h}|\\
&=&|(\nabla\cdot\tilde\bv,\;\rho)_{\T_h}+\l\{\bQ_h\tilde\bv-\tilde\bv\},\rho\bn\r_{\partial\T_h}|\\
&\ge|&(\nabla\cdot\tilde\bv,\;\rho)_{\T_h}|-C_1h\|\tilde{\bv}\|_1\|\rho\|\\
&\ge&(C_0-C_1h)\|\tilde{\bv}\|_1\|\rho\|.
\end{eqnarray*}
Using the above equation and (\ref{m9}), we have
\begin{eqnarray*}
\frac{|(\nabla_w\cdot\bv,\rho)|} {\3bar\bv\3bar} &\ge &
\frac{(C_0-C_1h)\|\tilde{\bv}\|_1\|\rho\|}{C_0\|\tilde\bv\|_1}\ge
\beta\|\rho\|
\end{eqnarray*}
for a positive constant $\beta$. This completes the proof of the lemma.
\end{proof}

\begin{lemma}
The weak Galerkin method (\ref{wg1})-(\ref{wg2}) has a unique solution.
\end{lemma}

\smallskip

\begin{proof}
It suffices to show that zero is the only solution of
(\ref{wg1})-(\ref{wg2}) if $\bbf=0$. To this end, let $\bbf=0$ and
take $\bv=\bu_h$ in (\ref{wg1}) and $q=p_h$ in (\ref{wg2}). By adding the
two resulting equations, we obtain
\[
(\nabla_w\bu_h,\ \nabla_w\bu_h)_{\T_h}=0,
\]
which implies that $\nabla_w \bu_h=0$ on each element $T$. By (\ref{happy}), we have $\|\bu_h\|_{1,h}=0$ which implies that $\bu_h=0$.

Since $\bu_h=0$ and $\bbf=0$, the equation (\ref{wg1}) becomes $(\nabla_w\bv,\ p_h)=0$ for any $\bv\in V_h$. Then the inf-sup condition (\ref{inf-sup}) implies $p_h=0$. We have proved the lemma.
\end{proof}

\section{Error Equations}
In this section, we will derive the equations that the errors satisfy. 
Let $\be_h=\bQ_h\bu-\bu_h$, $\bepsilon_h=\bu-\bu_h$ and $\varepsilon_h=Q_hp-p_h$.

\begin{lemma}
For any $\bv\in V_h$ and $q\in W_h$, the following error equations hold true,
\begin{eqnarray}
(\nabla_w\be_h,\; \nabla_w\bv)_{\T_h}-(\varepsilon_h,\;\nabla_w\cdot\bv)&=&\ell_1(\bu,\bv)-\ell_2(\bu,\bv)-\ell_3(p,\bv),\label{ee1}\\
(\nabla_w\cdot\be_h,\ q)&=&-\ell_4(\bu,q),\label{ee2}
\end{eqnarray}
where
\begin{eqnarray}
\ell_1(\bu,\ \bv)&=&\l  (\nabla\bu-\Q_h(\nabla\bu))\cdot\bn,\;\bv-\{\bv\}\r_{\partial\T_h},\label{l1}\\
\ell_2(\bu,\bv)&=&(\nabla_w (\bu-\bQ_h\bu),\nabla_w\bv)_{\T_h},\label{l3}\\
\ell_3(p,\bv)&=&\l p-Q_hp, (\bv-\{\bv\})\cdot\bn\r_{\partial\T_h},\label{l5}\\
\ell_4(\bu,q)&=&\l\{\bu-\bQ_h\bu\}\cdot\bn,q\r_{\partial\T_h}.\label{l7}
\end{eqnarray}
\end{lemma}

\begin{proof}
First, we test (\ref{moment}) by
$\bv\in V_h$ to obtain
\begin{equation}\label{mm0}
-(\Delta\bu,\;\bv)+(\nabla p,\ \bv)=(\bbf,\; \bv).
\end{equation}
Integration by parts and the fact $\langle
\nabla \bu\cdot\bn,\;\{\bv\}\rangle_{\partial\T_h}=0$ give
\begin{equation}\label{mm1}
-(\Delta\bu,\;\bv)=(\nabla \bu,\nabla \bv)_{\T_h}- \langle
\nabla \bu\cdot\bn,\bv-\{\bv\}\rangle_{\partial\T_h}.
\end{equation}
It follows from integration by parts, (\ref{wg}) and (\ref{key}),
\begin{eqnarray}
(\nabla \bu,\nabla \bv)_{\T_h}&=&(\Q_h\nabla  \bu,\nabla \bv)_{\T_h}\nonumber\\
&=&-(\bv,\nabla\cdot (\Q_h\nabla \bu))_{\T_h}+\langle \bv, \Q_h\nabla \bu\cdot\bn\rangle_{\partial\T_h}\nonumber\\
&=&(\Q_h\nabla \bu, \nabla_w \bv)_{\T_h}+\langle \bv-\{\bv\},\Q_h\nabla \bu\cdot\bn\rangle_{\partial\T_h}\nonumber\\
&=&( \nabla_w \bu, \nabla_w \bv)_{\T_h}+\langle \bv-\{\bv\},\Q_h\nabla \bu\cdot\bn\rangle_{\partial\T_h}.\label{j1}
\end{eqnarray}
Combining (\ref{mm1}) and (\ref{j1}) gives
\begin{eqnarray}
-(\Delta\bu,\;\bv)&=&(\nabla_w \bu,\nabla_w\bv)_{\T_h}-\ell_1(\bu,\bv).\label{mm2}
\end{eqnarray}
Using integration by parts and $\bv\in V_h$ and (\ref{wd}),  we have
\begin{eqnarray}
(\nabla p,\ \bv)&=& -(p,\nabla\cdot\bv)_{\T_h}+\l p, \bv\cdot\bn\r_{\partial\T_h}\nonumber\\
&=&-(Q_hp,\nabla\cdot\bv)_{\T_h}+\l p, (\bv-\{\bv\})\cdot\bn\r_{\partial\T_h}\nonumber\\
&=&(\nabla Q_hp,\;\bv)_{\T_h}-\l Q_hp, \bv\cdot\bn\r_{\partial\T_h}+\l p, (\bv-\{\bv\})\cdot\bn\r_{\partial\T_h}\nonumber\\
&=&-(Q_hp,\nabla_w\cdot\bv)_{\T_h}-\l Q_hp,\; (\bv-\{\bv\})\cdot\bn\r_{\partial\T_h}+\l p,\; (\bv-\{\bv\})\cdot\bn\r_{\partial\T_h}\nonumber\\
&=&-(Q_hp,\nabla_w\cdot\bv)_{\T_h}+\ell_3(p,\bv).\label{mm3}
\end{eqnarray}
Substituting (\ref{mm2}) and (\ref{mm3}) into (\ref{mm0}) gives
\begin{equation}\label{mm4}
(\nabla_w\bu,\nabla_w\bv)_{\T_h}-(Q_hp,\nabla_w\cdot\bv)_{\T_h}=(\bbf,\bv)+\ell_1(\bu,\bv)-\ell_3(p,\bv).
\end{equation}
The difference of (\ref{mm4}) and (\ref{wg1}) implies
\begin{equation}\label{mm10}
(\nabla_w\bepsilon_h,\nabla_w\bv)_{\T_h}-(\varepsilon_h,\nabla_w\cdot\bv)_{\T_h}=\ell_1(\bu,\bv)-\ell_3(p,\bv)\quad\forall\bv\in V_h.
\end{equation}
Adding and subtracting $(\nabla_w\bQ_h\bu,\nabla_w\bv)_{\T_h}$ in (\ref{mm10}), we have
\begin{equation}\label{mm11}
(\nabla_w\be_h,\nabla_w\bv)_{\T_h}-(\varepsilon_h,\nabla_w\cdot\bv)_{\T_h}=\ell_1(\bu,\bv)-\ell_2(\bu,\bv)-\ell_3(p,\bv),
\end{equation}
which implies (\ref{ee1}).

Testing equation (\ref{cont}) by $q\in W_h$ and using (\ref{wd}) give
\begin{eqnarray*}
(\nabla\cdot\bu,\;q)&=&-(\bu,\;\nabla q)_{\T_h}+\l\bu\cdot\bn,q\r_{\partial\T_h}\\
&=&-(\bQ_h\bu,\;\nabla q)_{\T_h}+\l\bu\cdot\bn,q\r_{\partial\T_h}\\
&=&(\nabla_w\cdot\bQ_h\bu,\;q)_{\T_h}+\l\{\bu-\bQ_h\bu\}\cdot\bn,\;q\r_{\partial\T_h}\\
&=&(\nabla_w\cdot\bQ_h\bu,\;q)_{\T_h}+\ell_4(\bu,q).
\end{eqnarray*}
which implies
\begin{eqnarray}
(\nabla_w\cdot\bQ_h\bu,\;q)_{\T_h}=-\ell_4(\bu,q).\label{mm5}
\end{eqnarray}
The difference of (\ref{mm5}) and (\ref{wg2}) implies (\ref{ee2}). We have proved the lemma.

\end{proof}

\section{Error Estimates in Energy Norm}
In this section, we shall establish optimal order error estimates
for the velocity approximation $\bu_h$ in $\3bar\cdot\3bar$ norm  and for the pressure approximation $p_h$ in
the standard $L^2$ norm. 

It is easy to see that the following equations hold true for $\{\bv\}$ defined in (\ref{avg}),
\begin{equation}\label{jp}
\|\bv-\{\bv\}\|_e=\|[\bv]\|_e\quad {\rm if} \;e\subset \partial\Omega,\quad\|\bv-\{\bv\}\|_e=\frac12\|[\bv]\|_e\;\; {\rm if} \;e\in\E_h^0.
\end{equation}

\begin{lemma}
Let $(\bw,\rho)\in \bH^{k+1}(\Omega)\times H^k(\Omega)$ and
$(\bv,q)\in V_h\times W_h$. Assume that the finite element partition $\T_h$ is
shape regular. Then, the following estimates hold true
\begin{eqnarray}
|\ell_1(\bw,\ \bv)|&\le& Ch^{k}|\bw|_{k+1}\3bar \bv\3bar,\label{mmm2}\\
|\ell_2(\bw,\ \bv)|&\le& Ch^{k}|\bw|_{k+1}\3bar \bv\3bar,\label{mmm3}\\
|\ell_3(\rho,\ \bv)|&\le& Ch^{k}|\rho|_{k}\3bar \bv\3bar.\label{mmm4}\\
|\ell_4(\bw,q)&\le& Ch^{k}|\bw|_{k+1}\|q\|.\label{mmm5}
\end{eqnarray}
\end{lemma}

\begin{proof}
Using the Cauchy-Schwarz inequality, the trace inequality (\ref{trace}), (\ref{jp}) and  (\ref{happy}), we have
\begin{eqnarray*}
|\ell_1(\bw,\ \bv)|&=&|\sum_{T\in\T_h}\l\bv-\{\bv\},\ \nabla\bw\cdot\bn-\Q_h(\nabla\bw)\cdot\bn\r_\pT|\\
&\le & C \sum_{T\in\T_h}\|\nabla \bw-\Q_h\nabla \bw\|_{\pT}
\|\bv-\{\bv\}\|_\pT\nonumber\\
&\le & C \left(\sum_{T\in\T_h}h_T\|(\nabla \bw-\Q_h\nabla \bw)\|_{\pT}^2\right)^{\frac12}
\left(\sum_{e\in\E_h}h_e^{-1}\|[\bv]\|_e^2\right)^{\frac12}\\
&\le & Ch^{k}|\bw|_{k+1}\3bar \bv\3bar,
\end{eqnarray*}
It follows from (\ref{wg}), integration  by parts, (\ref{trace}) and (\ref{jp}) that for any $\bq\in [P_{j}(T)]^{d\times d}$,
\begin{eqnarray}
|(\nabla_w(\bw-\bQ_h\bw), \bq)_{T}|&=&|-(\bw-\bQ_h\bw, \nabla\cdot\bq)_{T}+\l \bw-\{\bQ_h\bw\}, \bq\cdot\bn\r_{\pT}|\nonumber\\
&=&|(\nabla (\bw-\bQ_h\bw), \bq)_{T}+\l \bQ_h\bw-\{\bQ_h\bw\}, \bq\cdot\bn\r_{\pT}|\nonumber\\
&\le& \|\nabla (\bw-\bQ_h\bw)\|_T\|\bq\|_T+Ch^{-1/2}\|[\bQ_h\bw]\|_\pT\|\bq\|_T\nonumber\\
&\le& \|\nabla (\bw-\bQ_h\bw)\|_T\|\bq\|_T+Ch^{-1/2}\|[\bw-\bQ_h\bw]\|_\pT\|\bq\|_T\nonumber\\
&\le& Ch^k|\bw|_{k+1, T}\|\bq\|_T.\label{m30}
\end{eqnarray}
Letting $\bq=\nabla_w\bv$ in (\ref{m30}) and taking summation over $T$, we have
\begin{eqnarray*}
|\ell_2(\bw,\ \bv)|&=&|(\nabla_w (\bw-\bQ_h\bw),\nabla_w\bv)_{\T_h}|\\
&\le& Ch^{k}|\bw|_{k+1}\3bar \bv\3bar.
\end{eqnarray*}
It follows from the definition of $Q_h$, (\ref{trace}), (\ref{jp}) and (\ref{happy}) that
\begin{eqnarray*}
|\ell_3(\rho,\ \bv)|&=&|\sum_{T\in\T_h}\l \rho-Q_h\rho, \bv\cdot\bn-\{\bv\}\cdot\bn\r_\pT|\\
&\le & C \sum_{T\in\T_h}\|\rho-Q_h\rho\|_{\pT}
\|\bv-\{\bv\}\|_\pT\nonumber\\
&\le & C \left(\sum_{T\in\T_h}h_T\|\rho-Q_h\rho\|_{\pT}^2\right)^{\frac12}
\left(\sum_{e\in\E_h}h_e^{-1}\|[\bv]\|_e^2\right)^{\frac12}\\
&\le & Ch^{k}|\rho|_{k+1}\3bar \bv\3bar.
\end{eqnarray*}
Similarly we have
\begin{eqnarray*}
|\ell_4(\bw,q)|&=&|\sum_{T\in\T_h}\l\{\bw-\bQ_h\bw\}\cdot\bn,q\r_\pT|\\
&\le & C \sum_{T\in\T_h}\|\bw-\bQ_h\bw\|_{\pT}
\|q\|_\pT\nonumber\\
&\le & C \left(\sum_{T\in\T_h}h_T^{-1}\|\bw-\bQ_h\bw\|_{\pT}^2\right)^{\frac12}
\left(\sum_{T\in\T_h}h_T\|q\|_e^2\right)^{\frac12}\\
&\le & Ch^{k}|\bw|_{k+1}\|q\|.
\end{eqnarray*}
We have proved the lemma.
\end{proof}

\medskip

\begin{theorem}\label{h1-bd}
Let  $(\bu_h,p_h)\in
V_h\times W_h$ be the solution of 
(\ref{wg1})-(\ref{wg2}). Then, the following error
estimates hold true
\begin{eqnarray}
\3bar \bQ_h\bu-\bu_h\3bar &\le& Ch^{k}(|\bu|_{k+1}+|p|_k),\label{errv}\\
\|Q_hp-p_h\|&\le& Ch^{k}(|\bu|_{k+1}+|p|_k).\label{errp}
\end{eqnarray}
\end{theorem}

\smallskip

\begin{proof}
It follows from (\ref{ee1}) that for any $\bv\in V_h$
\begin{eqnarray}
|(\varepsilon_h,\;\nabla_w\cdot\bv)_{\T_h}|&=&|(\nabla_w\be_h,\; \nabla_w\bv)_{\T_h}-\ell_1(\bu,\bv)+\ell_2(\bu,\bv)+\ell_3(p,\bv)|\nonumber\\
&\le& C(\3bar\be_h\3bar +h^k|\bu|_{k+1})\3bar\bv\3bar.\label{e10}
\end{eqnarray}
Then the  estimate (\ref{e10}) and (\ref{inf-sup}) yield
\begin{eqnarray}
\|\varepsilon_h\|&\le& C(\3bar\be_h\3bar +h^k|\bu|_{k+1}).\label{e11}
\end{eqnarray}

By letting $\bv=\be_h$ in (\ref{ee1}) and $q=\varepsilon_h$ in
(\ref{ee2}) and adding the two resulting equations, we have
\begin{eqnarray}
\3bar \be_h\3bar^2&=&|\ell_1(\bu,\be_h)-\ell_2(\bu,\be_h)-\ell_3(p,\be_h)+\ell_4(\bu,\varepsilon_h)|.\label{main}
\end{eqnarray}
It then follows from (\ref{mmm2})-(\ref{mmm5}) and (\ref{e11}) that
\begin{eqnarray}
\3bar \be_h\3bar^2 &\le& Ch^{k}(|\bu|_{k+1}+|p|_k)\3bar \be_h\3bar+Ch^k|\bu|_{k+1}\|\varepsilon_h\|\nonumber\\
&\le& Ch^{k}(|\bu|_{k+1}+|p|_k)\3bar \be_h\3bar+Ch^k|\bu|_{k+1}(\3bar\be_h\3bar+Ch^k|\bu|_{k+1})\nonumber\\
&\le& Ch^{2k}(|\bu|^2_{k+1}+|p|^2_k)+\frac12\3bar\be_h\3bar^2, \label{b-u}
\end{eqnarray}
which implies (\ref{errv}). The pressure error estimate (\ref{errp}) follows immediately from (\ref{e11}) and (\ref{errv}).
\end{proof}

\section{Error Estimates in $L^2$ Norm}

In this section, we shall derive an $L^2$-error estimate
for the velocity approximation through a duality argument. Recall that $\be_h=\bQ_h\bu-\bu_h$ and $\bepsilon_h=\bu-\bu_h$. To this
end, consider the problem of seeking $(\bpsi,\xi)$ such that
\begin{eqnarray}
-\Delta\bpsi+\nabla \xi&=\bepsilon_h &\quad \mbox{in}\;\Omega,\label{dual-m}\\
\nabla\cdot\bpsi&=0 &\quad\mbox{in}\;\Omega,\label{dual-c}\\
\bpsi&= 0 &\quad\mbox{on}\;\partial\Omega.\label{dual-bc}
\end{eqnarray}
Assume that the dual problem has the $\bH^{2}(\Omega)\times
H^1(\Omega)$-regularity property in the sense that the solution
$(\bpsi,\xi)\in \bH^{2}(\Omega)\times H^1(\Omega)$ and the
following a priori estimate holds true:
\begin{equation}\label{reg}
\|\bpsi\|_{2}+\|\xi\|_1\le C\|\bepsilon_h\|.
\end{equation}

\medskip
\begin{theorem}
Let  $(\bu_h,p_h)\in V_h\times W_h$ be the solution of
(\ref{wg1})-(\ref{wg2}). Assume that (\ref{reg}) holds true.  Then we have
\begin{equation}\label{l2err}
\|\bu-\bu_h\|\le Ch^{k+1}(|\bu|_{k+1}+|p|_{k}).
\end{equation}
\end{theorem}

\begin{proof}
Testing (\ref{dual-m}) by $\bepsilon_h$   gives
\begin{eqnarray}
(\bepsilon_h, \bepsilon_h)&=&-(\Delta\bpsi,\;\bepsilon_h)+(\nabla \xi,\ \bepsilon_h).\label{m20}
\end{eqnarray}

Using integration by parts and the fact $\l \nabla\bpsi\cdot\bn,\{\bepsilon_h\}\r_{\partial\T_h}=0$, then
\begin{eqnarray*}
-(\Delta\bpsi,\bepsilon_h)
&=&(\nabla \bpsi,\ \nabla\bepsilon_h)_{\T_h}-\l
\nabla\bpsi\cdot\bn,\ \bepsilon_h- \{\bepsilon_h\}\r_{\partial\T_h}\\
&=&(\Q_h\nabla \bpsi,\ \nabla\bepsilon_h)_{\T_h}+(\nabla\bpsi-\Q_h\nabla \bpsi,\ \nabla\bepsilon_h)_{\T_h}\\
& & -\l
\nabla\bpsi\cdot\bn,\ \bepsilon_h- \{\bepsilon_h\}\r_{\partial\T_h}\\
&=&-(\nabla\cdot\Q_h\nabla \bpsi,\ \bepsilon_h)_{\T_h}+\l \Q_h\nabla\bpsi\cdot\bn,\ \bepsilon_h\r_{\partial\T_h}\\
& &+(\nabla\bpsi-\Q_h\nabla \bpsi,\ \nabla\bepsilon_h)_{\T_h}-\l\nabla\bpsi\cdot\bn,\ \bepsilon_h- \{\bepsilon_h\}\r_{\partial\T_h}\\
&=&(\Q_h\nabla \bpsi,\ \nabla_w\bepsilon_h)_{\T_h}+\l\Q_h\nabla\bpsi\cdot\bn,\ \bepsilon_h-\{\bepsilon_h\}\r_{\partial\T_h}\\
& &+(\nabla\bpsi-\Q_h\nabla \bpsi,\ \nabla\bepsilon_h)_{\T_h}-\l\nabla\bpsi\cdot\bn,\ \bepsilon_h- \{\bepsilon_h\}\r_{\partial\T_h}\\
&=&(\Q_h\nabla \bpsi,\ \nabla_w\bepsilon_h)_{\T_h}+(\nabla\bpsi-\Q_h\nabla \bpsi,\ \nabla\bepsilon_h)_{\T_h}-\ell_1(\bpsi,\bepsilon_h).
\end{eqnarray*}
It follows from (\ref{key}) that
\begin{eqnarray*}
(\Q_h\nabla \bpsi,\ \nabla_w\bpsi_h)_{\T_h}&=&(\nabla_w \bpsi,\;\nabla_w\bepsilon_h)_{\T_h}\\
&=&(\nabla_w \bQ_h\bpsi,\;\nabla_w\bepsilon_h)_{\T_h}+(\nabla_w (\bpsi-\bQ_h\bpsi),\;\nabla_w\bepsilon_h)_{\T_h}.
\end{eqnarray*}
The equation (\ref{mm5}) implies
\begin{equation}\label{mmm44}
(\varepsilon_h,\nabla_w\cdot\bQ_h\bpsi)_{\T_h}=-\ell_4(\bpsi,\varepsilon_h).
\end{equation}
Using the equation (\ref{mm10}) and (\ref{mmm44}), we have
\begin{eqnarray*}
(\nabla_w \bQ_h\bpsi,\;\nabla_w\bepsilon_h)_{\T_h}=\ell_1(\bu,\bQ_h\bpsi)-\ell_3(p,\bQ_h\bpsi)-\ell_4(\bpsi,\varepsilon_h).\label{m42}
\end{eqnarray*}
Combining the three equations above imply that
\begin{eqnarray}
-(\Delta\bpsi,\bepsilon_h)&=&\ell_1(\bu,\bQ_h\bpsi)-\ell_3(p,\bQ_h\bpsi)-\ell_4(\bpsi,\varepsilon_h)+(\nabla_w (\bpsi-\bQ_h\bpsi),\;\nabla_w\bepsilon_h)_{\T_h}\nonumber\\
&+&(\nabla\bpsi-\bQ_h\nabla \bpsi,\ \nabla\bepsilon_h)_{\T_h}-\ell_1(\bpsi,\bepsilon_h).\label{m40}
\end{eqnarray}

It follows from integration by parts and (\ref{cont}), (\ref{wd}) and (\ref{wg2}) that
\begin{eqnarray}
(\nabla \xi,\ \bepsilon_h)&=&(\nabla \xi,\ \bu)-(\nabla \xi,\ \bu_h)=-(\nabla \xi,\ \bu_h)\nonumber\\
&=&(Q_h\xi,\ \nabla\cdot\bu_h)_{\T_h}-\l \xi,\; \bu_h\cdot\bn-\{\bu_h\}\cdot\bn\r_{\partial\T_h}\nonumber\\
&=&-(\nabla Q_h\xi,\ \bu_h)_{\T_h}+\l Q_h\xi,\; \bu_h\cdot\bn\r_{\partial\T_h}
  \nonumber \\ & & -\l \xi,\; \bu_h\cdot\bn-\{\bu_h\}\cdot\bn\r_{\partial\T_h}\nonumber\\
&=&(Q_h\xi,\ \nabla_w\cdot\bu_h)_{\T_h}+\l Q_h\xi,\; \bu_h\cdot\bn-\{\bu_h\}\cdot\bn\r_{\partial\T_h}
 \nonumber \\ & & -\l \xi,\; \bu_h\cdot\bn-\{\bu_h\}\cdot\bn\r_{\partial\T_h}\nonumber\\
&=&-\l \xi-Q_h\xi,\; (\bu_h-\{\bu_h\})\cdot\bn\r_{\partial\T_h}\nonumber\\
&=&-\ell_3(\xi,\bu_h)=\ell_3(\xi,\bepsilon_h).\label{m41}
\end{eqnarray}
Combining (\ref{m20})-(\ref{m42}), we have
\begin{eqnarray}
\|\bepsilon_h\|^2&=&\ell_1(\bu,\bQ_h\bpsi)-\ell_3(p,\bQ_h\bpsi_h)-\ell_4(\bpsi,\varepsilon_h)+(\nabla_w(\bpsi-\bQ_h\bpsi),\;\nabla_w\bepsilon_h)\nonumber\\
&+&(\nabla\bpsi-\Q_h\nabla \bpsi,\ \nabla\bepsilon_h)_{\T_h}-\ell_1(\bpsi, \bepsilon_h)+\ell_3(\xi,\bepsilon_h).\label{mmm22}
\end{eqnarray}

Next we will estimate all the terms on the right hand side of (\ref{mmm22}). Using the Cauchy-Schwarz inequality, the trace inequality (\ref{trace}) and the definitions of $\bQ_h$ and $\Q_h$
we obtain
\begin{eqnarray*}
|\ell_1(\bu,\bQ_h\bpsi)|&\le&\left| \langle (\nabla \bu-\Q_h\nabla
\bu)\cdot\bn,\;
\bQ_h\bpsi-\{\bQ_h\bpsi\}\rangle_{\partial\T_h} \right|\\
&\le& \left(\sum_{T\in\T_h}\|(\nabla \bu-\Q_h\nabla \bu)\|^2_\pT\right)^{1/2}
\left(\sum_{T\in\T_h}\|\bQ_h\bpsi-\{\bQ_h\bpsi\}\|^2_\pT\right)^{1/2}\nonumber \\
&\le& C\left(\sum_{T\in\T_h}h_T\|(\nabla \bu-\Q_h\nabla \bu)\|^2_\pT\right)^{1/2}
\left(\sum_{T\in\T_h}h_T^{-1}\|[\bQ_h\bpsi-\bpsi]\|^2_\pT\right)^{1/2} \nonumber\\
&\le&  Ch^{k+1}|\bu|_{k+1}\|\bpsi\|_2.\nonumber
\end{eqnarray*}
Similarly, we have
\begin{eqnarray*}
|\ell_3(p,\bQ_h\bpsi)|&\le&\left| \langle p-Q_hp,\;
(\bQ_h\bpsi-\{\bQ_h\bpsi\})\cdot\bn\rangle_{\partial\T_h} \right|\\
&\le& \left(\sum_{T\in\T_h}h_T\|p-Q_hp\|^2_\pT\right)^{1/2}
\left(\sum_{T\in\T_h}h_T^{-1}\|\bQ_h\bpsi-\{\bQ_h\bpsi\}\|^2_\pT\right)^{1/2}\nonumber \\
&\le&  Ch^{k+1}|p|_{k}\|\bpsi\|_2.\nonumber
\end{eqnarray*}

It follows from (\ref{m30}) and (\ref{errv}) that
\begin{eqnarray*}
|(\nabla_w(\bpsi-\bQ_h\bpsi),\;\nabla_w \bepsilon_h)_{\T_h}|&\le& C\3bar \bepsilon_h\3bar \3bar \bpsi-\bQ_h\bpsi\3bar\\
&\le& Ch^{k+1}(|\bu|_{k+1}+|p|_k)|\bpsi|_2.
\end{eqnarray*}
The estimates (\ref{happy}) and (\ref{errv}) imply
\begin{eqnarray*}
|(\nabla\bpsi-\Q_h\nabla \bpsi,\ \nabla \bepsilon_h)_{\T_h}|&\le& C(\sum_{T\in\T_h}\|\nabla \bepsilon_h\|_T^2)^{1/2} (\sum_{T\in\T_h}\|\nabla\bpsi-\Q_h\nabla \bpsi\|_T^2)^{1/2}\\
&\le& C(\sum_{T\in\T_h}(\|\nabla(\bu-\bQ_h\bu)\|_T^2+\|\nabla(\bQ_h\bu-\bu_h)\|_T^2))^{1/2}\\
&\times &(\sum_{T\in\T_h}\|\nabla\bpsi-\Q_h\nabla \bpsi\|_T^2)^{1/2}\\
%&\le& C(\sum_{T\in\T_h}(\|u-Q_hu\|_T^2+\|Q_hu-u_h\|_T^2)^{1/2} (\sum_{T\in\T_h}\|\Phi-Q_h\Phi\|_T^2)^{1/2}\\
&\le& Ch|\bpsi|_2(h^k|\bu|_{k+1}+\3bar \be_h\3bar)\\
%&\le& Ch|\bpsi|_2(h^k|\bu|_{k+1}+\3bar \bu-\bu_h\3bar+\3bar \bQ_h\bu-\bu\3bar)\\
&\le& Ch^{k+1}(|\bu|_{k+1}+|p|_k)|\bpsi|_2.
\end{eqnarray*}
Using (\ref{happy}) and (\ref{errv}), we obtain
\begin{eqnarray*}
|\ell_1(\bpsi,\bepsilon_h)|&=&\left|\l (\nabla \bpsi-\Q_h\nabla \bpsi)\cdot\bn,\
\bepsilon_h-\{\bepsilon_h\}\r_{\partial\T_h}\right|\\
&\le&\sum_{T\in\T_h} h_T^{1/2}\|\nabla \bpsi-\Q_h\nabla \bpsi\|_\pT  h_T^{-1/2}\|[\bepsilon_h]\|_\pT\\
&\le&Ch\|\bpsi\|_2(\sum_{T\in\T_h}h_T^{-1}(\|[\be_h]\|^2_\pT+\|[\bu-\bQ_h\bu]\|^2_\pT))^{1/2}\\
&\le&Ch\|\bpsi\|_2(\3bar\be_h\3bar+(\sum_{T\in\T_h}h_T^{-1}\|[\bu-\bQ_h\bu]\|^2_\pT)^{1/2}\\
%&\le&Ch\|\bpsi\|_2(\3bar \be_h\3bar+\3bar \bu-\bQ_h\bu\3bar+Ch^k|\bu|_{k+1})\\
&\le&  Ch^{k+1 }(|\bu|_{k+1}+|p|_k)\|\bpsi\|_2.
\end{eqnarray*}
Similarly, we have
\begin{eqnarray*}
|\ell_3(\xi,\bepsilon_h)|&=&\left|\l \xi-Q_h\xi,\
(\bepsilon_h-\{\bepsilon_h\})\cdot\bn\r_{\partial\T_h}\right|\\
&\le&\sum_{T\in\T_h} h_T^{1/2}\|\xi-Q_h\xi\|_\pT  h_T^{-1/2}\|[\bepsilon_h]\|_\pT\\
&\le&  Ch^{k+1}(|\bu|_{k+1}+|p|_{k})\|\xi\|_1.
\end{eqnarray*}
Using (\ref{mmm5}) and (\ref{errp}), we have
\[
\ell_4(\bpsi,\varepsilon_h)\le Ch^{k+1}(|\bu|_{k+1}+|p|_{k})\|\bpsi\|_2.
\]
Combining all the estimates above
with (\ref{mmm22}) yields
$$
\|\bepsilon_h\|^2 \leq C h^{k+1}(|\bu|_{k+1}+|p|_k)(\|\bpsi\|_2+\|\xi\|_1).
$$
The estimate (\ref{l2err}) follows from the above inequality and
the regularity assumption (\ref{reg}). We have completed the proof.
\end{proof}

\section{Numerical Experiments}\label{Section:numerical-experiments}

%%%%%%%%%%%%%%%%%%%%%%%%%%%%%%
\subsection{Example 1}
Consider problem  \eqref{moment}--\eqref{bc} with $\Omega=(0,1)^2$.
The source term $\bf$ and the boundary value $\bg$ are chosen so that the exact solution is
\a{
    \bu(x,y)&=\p{\sin \pi y \\ \cos \pi x}, \quad
      p = \sin 2 \pi y .
}
In this example, we use uniform triangular grids shown in Figure \ref{grid1}.
In Table \ref{t1}, we list the errors and the orders of convergence.
We can see that the optimal order of convergence is achieved in all finite element
  methods.

\begin{figure}[h!]
 \begin{center} \setlength\unitlength{1.25pt}
\begin{picture}(260,80)(0,0)
  \def\tr{\begin{picture}(20,20)(0,0)\put(0,0){\line(1,0){20}}\put(0,20){\line(1,0){20}}
          \put(0,0){\line(0,1){20}} \put(20,0){\line(0,1){20}}
   \put(0,20){\line(1,-1){20}}   \end{picture}}
 {\setlength\unitlength{5pt}
 \multiput(0,0)(20,0){1}{\multiput(0,0)(0,20){1}{\tr}}}

  {\setlength\unitlength{2.5pt}
 \multiput(45,0)(20,0){2}{\multiput(0,0)(0,20){2}{\tr}}}

  \multiput(180,0)(20,0){4}{\multiput(0,0)(0,20){4}{\tr}}

 \end{picture}\end{center}
\caption{The first three levels of triangular grids for Example 1.}
\label{grid1}
\end{figure}
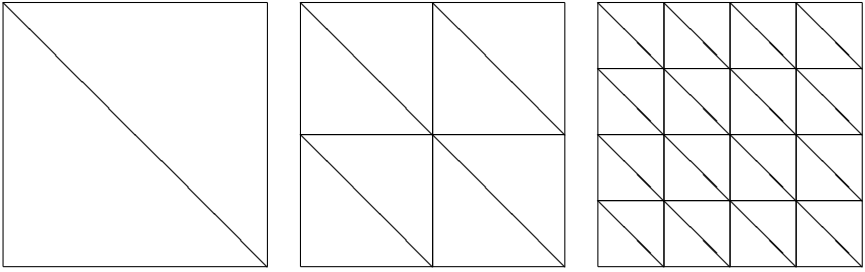

\begin{table}[h!]
  \centering   \renewcommand{\arraystretch}{1.05}
  \caption{Example 1:  Error profiles and convergence rates on grids shown in Figure \ref{grid1}. }
\label{t1}
\begin{tabular}{c|cc|cc|cc}
\hline
Grid & $\|\bu- \bu_h \|_0 $  &rate &  $\3bar \bu- \bu_h \3bar $ &rate
  &  $\|p - p_h \|_0 $ &rate   \\
\hline
 &\multicolumn{6}{c}{by the $P_1^2$-$P_0$ finite element } \\ \hline
 5&   0.9678E-03 &  1.90&   0.5352E-01 &  1.09&   0.3859E-01 &  1.76 \\
 6&   0.2530E-03 &  1.94&   0.2565E-01 &  1.06&   0.1257E-01 &  1.62 \\
 7&   0.6486E-04 &  1.96&   0.1250E-01 &  1.04&   0.4897E-02 &  1.36 \\
\hline
 &\multicolumn{6}{c}{by the $P_2^2$-$P_1$ finite element } \\ \hline
 4&   0.1598E-03 &  2.94&   0.1222E-01 &  1.95&   0.6123E-02 &  2.29 \\
 5&   0.2026E-04 &  2.98&   0.3091E-02 &  1.98&   0.1411E-02 &  2.12 \\
 6&   0.2544E-05 &  2.99&   0.7747E-03 &  2.00&   0.3384E-03 &  2.06 \\
\hline
 &\multicolumn{6}{c}{by the $P_3^2$-$P_2$ finite element } \\ \hline
 4&   0.5550E-05 &  4.01&   0.5624E-03 &  3.09&   0.1731E-02 &  2.75 \\
 5&   0.3409E-06 &  4.03&   0.6677E-04 &  3.07&   0.2267E-03 &  2.93 \\
 6&   0.2109E-07 &  4.01&   0.8113E-05 &  3.04&   0.2881E-04 &  2.98 \\
 \hline
 &\multicolumn{6}{c}{by the $P_4^2$-$P_3$ finite element } \\ \hline
 3&   0.6286E-05 &  5.24&   0.4739E-03 &  4.18&   0.4781E-03 &  4.90 \\
 4&   0.2057E-06 &  4.93&   0.3110E-04 &  3.93&   0.1493E-04 &  5.00 \\
 5&   0.6518E-08 &  4.98&   0.1978E-05 &  3.98&   0.6207E-06 &  4.59 \\
 \hline
\end{tabular}%
\end{table}%

%%%%%%%%%%%%%%%%%%%%%%%%%%%%%%
\subsection{Example 2}
Consider problem  \eqref{moment}--\eqref{bc} with $\Omega=(0,1)^2$.
The source term $\bf$ is chosen so that the exact solution is
\a{
    \bu(x,y)&=\p{-256(x-x^2)^2(y-y^2)(2-4y) \\ 256(x-x^2)(2-4x)(y-y^2)^2}, \quad
      p = x+y-1.
}
In this example, we use polygonal grids, consisting of dodecagons (12 sided polygons) and
   heptagons (7 sided polygons), shown in Figure \ref{grid-12}.
In Table \ref{t12}, we list the errors and the orders of convergence.
We can see that the optimal order of convergence is achieved in all finite element
  methods.

\begin{figure}[htb]\begin{center}\setlength\unitlength{1.5in} 
    \begin{picture}(3.2,1.4)
 \put(0,0){\includegraphics[width=1.5in]{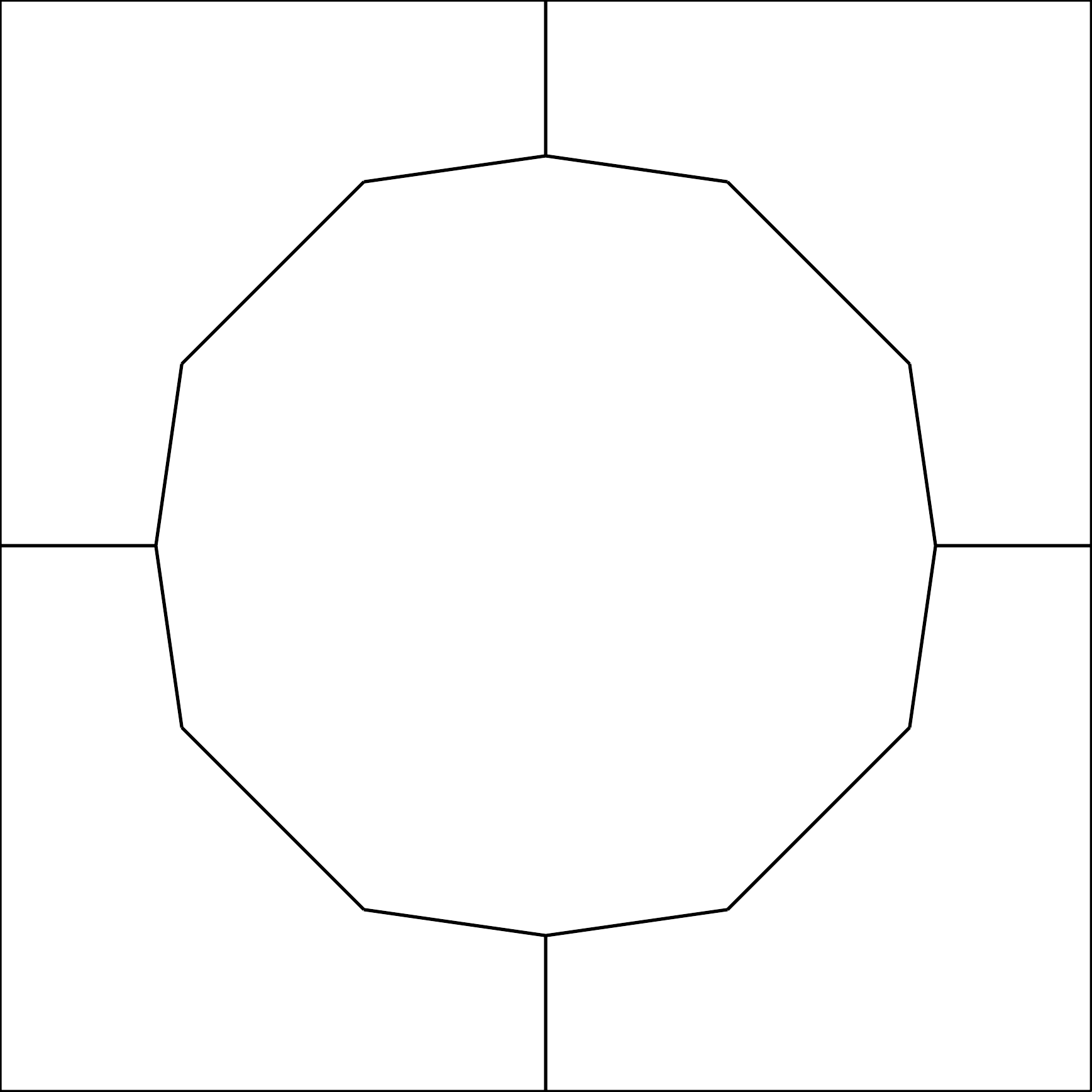}}  
 \put(1.1,0){\includegraphics[width=1.5in]{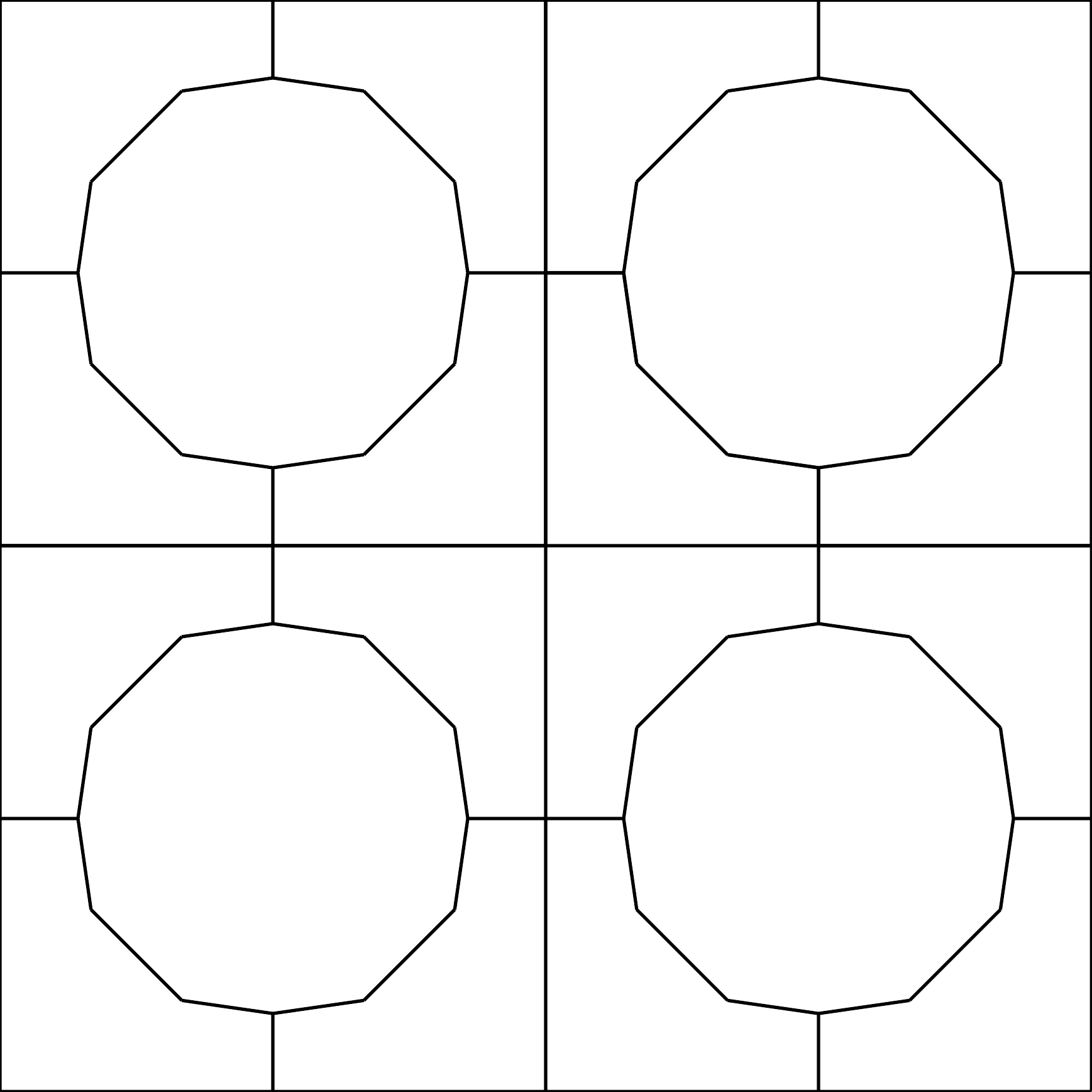}}
 \put(2.2,0){\includegraphics[width=1.5in]{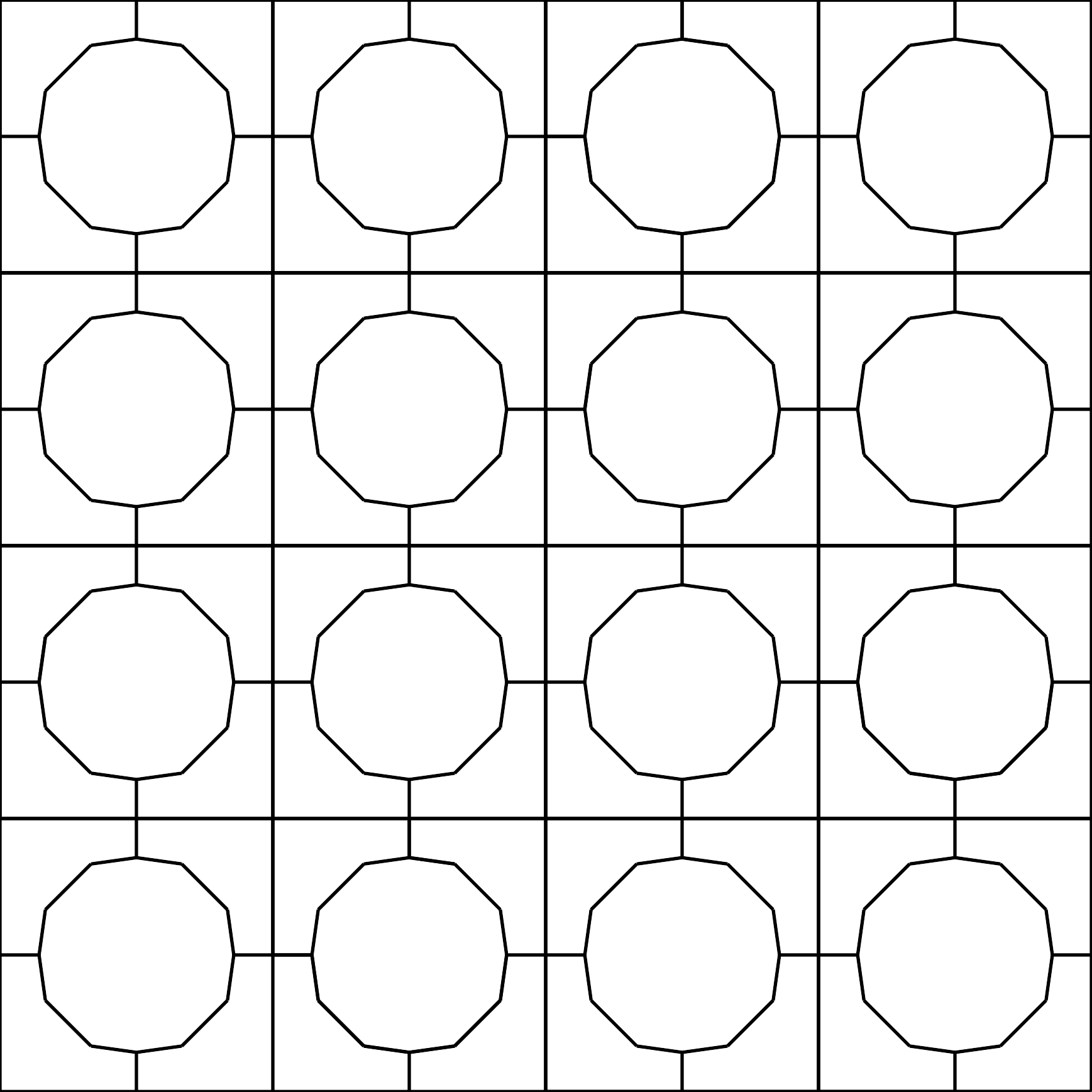}}
    \end{picture}
\caption{ The first three polygonal grids (consisting of dodecagons and heptagons)
   for the computation of Table \ref{t12} (Example 2).  } \label{grid-12}
\end{center}
\end{figure}

\begin{table}[h!]
  \centering   \renewcommand{\arraystretch}{1.05}
  \caption{Example 2:  Error profiles and convergence rates on grids shown
       in Figure \ref{grid-12}. }
\label{t12}
\begin{tabular}{c|cc|cc|cc}
\hline
Grid & $\|\bu- \bu_h \|_0 $  &rate &  $\3bar \bu- \bu_h \3bar $ &rate
  &  $\|p - p_h \|_0 $ &rate   \\
\hline
 &\multicolumn{6}{c}{by the $P_1^2$-$P_0$ finite element } \\ \hline
 4&   0.3844E-01 &  1.68&   0.7406E+00 &  0.93&   0.2329E+00 &  0.90 \\
 5&   0.1046E-01 &  1.88&   0.3741E+00 &  0.99&   0.9698E-01 &  1.26 \\
 6&   0.2708E-02 &  1.95&   0.1874E+00 &  1.00&   0.3982E-01 &  1.28 \\
\hline
 &\multicolumn{6}{c}{by the $P_2^2$-$P_1$ finite element } \\ \hline
3&   0.4929E-02 &  3.44&   0.2366E+00 &  2.20&   0.1052E+00 &  2.32 \\
 4&   0.6126E-03 &  3.01&   0.6352E-01 &  1.90&   0.2438E-01 &  2.11 \\
 5&   0.7694E-04 &  2.99&   0.1646E-01 &  1.95&   0.5619E-02 &  2.12 \\
\hline
 &\multicolumn{6}{c}{by the $P_3^2$-$P_2$ finite element } \\ \hline
 3&   0.3980E-03 &  3.76&   0.2366E-01 &  2.98&   0.1786E-01 &  2.67 \\
 4&   0.2812E-04 &  3.82&   0.3187E-02 &  2.89&   0.2590E-02 &  2.79 \\
 5&   0.1846E-05 &  3.93&   0.4061E-03 &  2.97&   0.3384E-03 &  2.94 \\
 \hline
 &\multicolumn{6}{c}{by the $P_4^2$-$P_3$ finite element } \\ \hline
 3&   0.3590E-04 &  4.81&   0.3079E-02 &  3.76&   0.2088E-02 &  3.76 \\
 4&   0.1173E-05 &  4.94&   0.2035E-03 &  3.92&   0.1316E-03 &  3.99 \\
 5&   0.3758E-07 &  4.96&   0.1300E-04 &  3.97&   0.7903E-05 &  4.06 \\
 \hline
\end{tabular}%
\end{table}%

\subsection{Example 3}
Consider problem \eqref{moment}--\eqref{bc} with $\Omega=(0,1)^3$.
The source term $\bf$ and the boundary value $\bg$ are chosen so that the exact solution is
\a{
    \bu(x,y)&=\p{y^4\\z^2\\x^2  }, \quad
       p = x-\frac 12.
}
We use tetrahedral meshes shown in Figure \ref{grid3d}.
The results of the 3D $P_k$-$P_{k+1}$ weak Galerkin finite element methods
    are listed in Table \ref{t2}.
The method is stable and is of optimal order convergence.

\begin{figure}[h!]
\begin{center}
 \setlength\unitlength{1pt}
    \begin{picture}(320,118)(0,3)
    \put(0,0){\begin{picture}(110,110)(0,0)
       \multiput(0,0)(80,0){2}{\line(0,1){80}}  \multiput(0,0)(0,80){2}{\line(1,0){80}}
       \multiput(0,80)(80,0){2}{\line(1,1){20}} \multiput(0,80)(20,20){2}{\line(1,0){80}}
       \multiput(80,0)(0,80){2}{\line(1,1){20}}  \multiput(80,0)(20,20){2}{\line(0,1){80}}
    \put(80,0){\line(-1,1){80}}\put(80,0){\line(1,5){20}}\put(80,80){\line(-3,1){60}}
      \end{picture}}
    \put(110,0){\begin{picture}(110,110)(0,0)
       \multiput(0,0)(40,0){3}{\line(0,1){80}}  \multiput(0,0)(0,40){3}{\line(1,0){80}}
       \multiput(0,80)(40,0){3}{\line(1,1){20}} \multiput(0,80)(10,10){3}{\line(1,0){80}}
       \multiput(80,0)(0,40){3}{\line(1,1){20}}  \multiput(80,0)(10,10){3}{\line(0,1){80}}
    \put(80,0){\line(-1,1){80}}\put(80,0){\line(1,5){20}}\put(80,80){\line(-3,1){60}}
       \multiput(40,0)(40,40){2}{\line(-1,1){40}}
        \multiput(80,40)(10,-30){2}{\line(1,5){10}}
        \multiput(40,80)(50,10){2}{\line(-3,1){30}}
      \end{picture}}
    \put(220,0){\begin{picture}(110,110)(0,0)
       \multiput(0,0)(20,0){5}{\line(0,1){80}}  \multiput(0,0)(0,20){5}{\line(1,0){80}}
       \multiput(0,80)(20,0){5}{\line(1,1){20}} \multiput(0,80)(5,5){5}{\line(1,0){80}}
       \multiput(80,0)(0,20){5}{\line(1,1){20}}  \multiput(80,0)(5,5){5}{\line(0,1){80}}
    \put(80,0){\line(-1,1){80}}\put(80,0){\line(1,5){20}}\put(80,80){\line(-3,1){60}}
       \multiput(40,0)(40,40){2}{\line(-1,1){40}}
        \multiput(80,40)(10,-30){2}{\line(1,5){10}}
        \multiput(40,80)(50,10){2}{\line(-3,1){30}}

       \multiput(20,0)(60,60){2}{\line(-1,1){20}}   \multiput(60,0)(20,20){2}{\line(-1,1){60}}
        \multiput(80,60)(15,-45){2}{\line(1,5){5}} \multiput(80,20)(5,-15){2}{\line(1,5){15}}
        \multiput(20,80)(75,15){2}{\line(-3,1){15}}\multiput(60,80)(25,5){2}{\line(-3,1){45}}
      \end{picture}}

    \end{picture}
    \end{center}
\caption{  The first three levels of grids used in Example 3. }
\label{grid3d}
\end{figure}
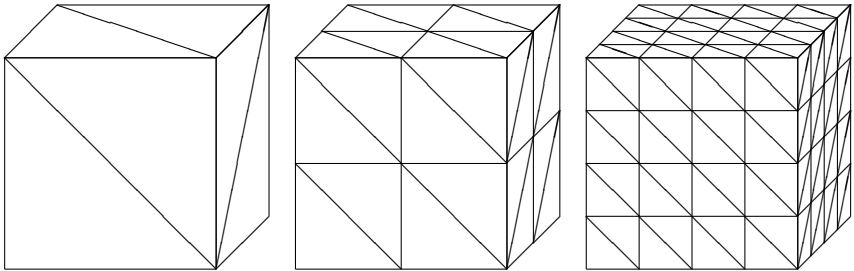

\begin{table}[h!]
  \centering   \renewcommand{\arraystretch}{1.05}
  \caption{Example 3:  Error profiles and convergence rates on grids shown in Figure \ref{grid3d}. }
\label{t2}
\begin{tabular}{c|cc|cc|cc}
\hline
Grid &  $\3bar \bu- \bu_h \3bar  $  &rate & $\|\bu- \bu_h \|_0 $ &rate
  &  $\|p - p_h \|_0 $ &rate   \\
\hline
 &\multicolumn{6}{c}{by the 3D $P_2^2$-$P_1$ finite element } \\ \hline
 1&    0.1845E+00& 0.00&    0.1289E-01& 0.00&    0.2125E+00& 0.00 \\
 2&    0.5331E-01& 1.79&    0.2383E-02& 2.44&    0.3299E-01& 2.69 \\
 3&    0.1422E-01& 1.91&    0.3475E-03& 2.78&    0.6230E-02& 2.40 \\
\hline
 &\multicolumn{6}{c}{by the 3D $P_3^2$-$P_2$ finite element } \\ \hline
 1&    0.3237E-01& 0.00&    0.1677E-02& 0.00&    0.2406E-01& 0.00 \\
 2&    0.4013E-02& 3.01&    0.1390E-03& 3.59&    0.2270E-02& 3.41 \\
 3&    0.5015E-03& 3.00&    0.9991E-05& 3.80&    0.2350E-03& 3.27 \\
 \hline
\end{tabular}%
\end{table}%

\end{document}